\documentclass[10pt]{article}
\usepackage{amsfonts}
\usepackage{amsmath}
\usepackage{mathrsfs}
\usepackage{color}
\usepackage{mathrsfs,amscd,amssymb,amsthm,amsmath,bm,graphicx,psfrag,subfigure,url}

\usepackage{indentfirst}

\def \[{\begin{equation}}
\def \]{\end{equation}}

\newtheorem{thm}{Theorem}[section]

\newtheorem{Case}{Case}

\newtheorem{fact}{Fact}
\newtheorem{lem}[thm]{Lemma}

\newtheorem{Obs}{Observation}
\newenvironment{kst}
{\setlength{\leftmargini}{2\parindent}
 \begin{itemize}
 \setlength{\itemsep}{-1.1mm}}
{\end{itemize}}

\newenvironment{wst}
{\setlength{\leftmargini}{1.5\parindent}
 \begin{itemize}
 \setlength{\itemsep}{-1.1mm}}
{\end{itemize}}

\begin{document}

\title{\bf On the maximum Zagreb indices of bipartite graphs with given connectivity}

\author{Xiaocong He$^{a,}$\footnote{Corresponding author},\ \  Xiaobo He$^b$}
\date{}

\maketitle

\begin{center}
$^a$ School of Mathematics and Statistics, Central South University, New Campus, \\ Changsha, Hunan, 410083, P.R. China\\[5pt]
hexc2018@qq.com (X.C.~He)
\medskip

$^b$ School of Architectural Engineering, Huanggang Normal University, \\ Huanggang, Hubei, 438000, P.R. China\\
1459313778@qq.com (X.B. He)
\end{center}

\begin{abstract}
The first Zagreb index $M_{1}$ of a graph is defined as the sum of the square of every vertex degree, and the second Zagreb index $M_{2}$ of a graph is defined as the sum of the product of vertex degrees of each pair of adjacent vertices. In this paper, we study the Zagreb indices of bipartite graphs of order $n$ with $\kappa(G)=k$ (resp. $\kappa'(G)=s$) and sharp upper bounds are obtained for $M_1(G)$ and $M_2(G)$ for $G\in \mathcal{V}^k_n$ (resp. $\mathcal{E}^s_n$), where $\mathcal{V}^k_n$ is the set of bipartite graphs of order $n$ with $\kappa(G)=k$, and $\mathcal{E}^s_n$ is the set of bipartite graphs of order $n$ with $\kappa'(G)=s$.
\end{abstract}

\vspace{2mm} \noindent{\bf Keywords}: First Zagreb index; Second Zagreb index; Connectivity; Edge connectivity

\vspace{2mm}

\setcounter{section}{0}
\section{Introduction}\setcounter{equation}{0}
In this paper,  we only consider simple and undirected graphs. Let $G=(V, E)$ be a graph with the vertex set $V(G)$ and the edge set $E(G)$. Then $G-v,\, G-uv$ denote the graph obtained from $G$ by deleting vertex $v \in V$, or edge $uv \in E$, respectively. This notation is naturally extended if more than one vertex or edge is deleted. Similarly, $G+uv$ is obtained from $G$ by adding an edge $uv\notin E(G)$. Denote by $|U|$ the cardinality of the set $U$. For a vertex subset $S$ of $V(G)$, $G[S]$ consists of $S$ and all edges in $E(G)$ whose endvertices are contained in $S$. If $G$ is connected and for a set $W$ of vertices (edges), $G-W$ is disconnected, then we say $W$ is a \textit{vertex (edge) cut set} of $G$. In order to formulate our results, we need some graph-theoretical notations and terminologies. For other undefined ones, the reader is refereed to \cite{06}.

For $u\in V(G)$, we denote its neighborhood and the degree by $N_G(u)$ and $d_G(u)$, respectively. Let $\delta(G)=\max\{d_G(u): u\in V(G)\}$. Recall that $G$ is called \textit{k-connected} if $|G|>k$ and $G-X$ is connected for every set $X\subseteq V(G)$ with $|X|<k$. The greatest integer $k$ such that $G$ is $k$-connected is the connectivity $\kappa(G)$ of $G$. Thus, $\kappa(G)=0$ if and only if $G$ is disconnected or $K_1$, and $\kappa(K_n)=n-1$ for all $n\geq1$. Analogously, if $|G|>1$ and $G-E'$ is connected for every edge set $E'\subseteq E(G)$ of fewer than $s$ edges, then $G$ is called \textit{s-edge-connected}. The greatest integer $s$ such that $G$ is $s$-connected is the edge-connectivity $\kappa'(G)$ of $G$. In particular, $\kappa'(G)=0$ if $G$ is disconnected. If $G$ is a graph of order $n$, we may have the following fact.
\begin{kst}
\item[{\rm (1)}] $\kappa(G)\leq\kappa'(G)\leq \delta(G)\leq n-1$, and
\item[{\rm (2)}] $\kappa(G)=n-1$, $\kappa'(G)=n-1$ and $G\cong K_n$ are equivalent.
\end{kst}

A bipartite graph $G$ is a simple graph, whose vertex set $V(G)$ can be partitioned into two disjoint subsets $X$ and $Y$ such that every edge of $G$ joins a vertex of $X$ with a vertex of $Y$. A bipartite graph in which every two vertices from different partition classes are adjacent is called complete, which is denoted by $K_{p,q}$, where $p=|X|$, $q=|Y|$. In particular, let $K_{p,0}$, ($p\geq1$) denote $pK_1$.

Let $\mathcal{V}^k_n$ (resp. $\mathcal{E}^s_n$) be the class of all $n$-vertex bipartite graphs with connectivity $k$ (resp. edge-connectivity $s$). Let $n,k$ and $r$ be three positive integers such that $n\geq r\geq k$. We define a graph $O_k\vee_1(K_1\bigcup K_{n-r-1,r-k})$, where $\bigcup$ is the union of two graphs, $O_k$ ($k\geq1$) is an empty graph of order $k$ and $\vee_1$ is a graph operation that join all the vertices in $O_k$ to the vertex of $K_1$ and to the vertices belonging to the partition of cardinality $n-r-1$ in $K_{n-r-1,r-k}$ respectively.

The \textit{first Zagreb index} $M_1=M_1(G)$ and the \textit{second Zagreb index} $M_2=M_2(G)$ of the graph $G$ are defined, respectively, as:
\begin{align*}
M_{1}=M_{1}(G)=\sum_{u\in V(G)}d^2(u),\ \ \ M_{2}=M_{2}(G)=\sum_{uv\in E(G)}d(u)d(v).
\end{align*}
In 1972, Gutman and Trinajsti\'c \cite{1972-Gutman-Trinajstic} discovered the quantities $M_1$ and $M_2$ in certain approximate
expressions for the total $\pi$-electron energy. The name \textit{Zagreb index} (or, \textit{Zagreb group index}) seems to be first
used in the survey article \cite{1983-Balaban} and after that became standard. In fact, these graph invariants
were proposed in 1975 as measures of branching of the carbon atom skeleton (see \cite{2}). In the past ten years the researchers mainly focused on the bounds or mathematical properties of $M_1, M_2$, one may be referred to \cite{07,2010-FLH,2011-HL,2011-LZ,2011-LZ1,2010-LZH,6,2010-LZ,7} and the references with in. Especially, Liu and Gutman \cite{LG} determined the upper bounds for Zagreb indices of connected graphs; Feng and Ili\'{c} \cite{2010-F} determined sharp upper bounds for Zagreb indices $M_1$ and $M_2$ of graphs with a given matching number; Li and Zhou \cite{2010-LZH} determined the upper and lower bounds for Zagreb indices of graphs with connectivity at most $k$; Li and Zhang determined sharp upper bounds for Zagreb indices of bipartite graphs with a given diameter; see \cite{2011-LZ}. For a survey of chemical applications
and mathematical properties of the Zagreb indices, we refer the reader to \cite{2017-BDG,2004-IGD,2013-IGT,2003-NKMT,7}.

Motivated by \cite{2010-LZH}, in this paper we determine the maximum values on $M_1$ and $M_2$ in the case of $n$-vertex bipartite graphs, which is an important class of graphs in graph theory. Based on the structure of bipartite graphs, sharp upper bounds on $M_1$ and $M_2$ among $\mathcal{V}^k_n$ (resp. $\mathcal{E}^s_n$) are determined. The corresponding extremal graphs are identified, respectively.
\section{Lemmas and results}
In this section, we shall determine the sharp upper bounds for $M_1(G)$ and $M_2(G)$ of bipartite graph $G\in \mathcal{V}^k_n$ (resp. $\mathcal{E}^s_n$).

By the definition of Zagreb indices, we obtain the following lemma immediately.
\begin{lem}\label{lem2.1}
Let $u,v$ be two non-adjacent vertices of a connected graph $G$. We have
$$M_1(G+uv)>M_1(G), \ \ \ M_2(G+uv)>M_2(G).$$
\end{lem}

\begin{lem}[\cite{2010-LZH}]\label{lem2.2}
Let $G$ be a connected graph and $u, v$ be two vertices of $G$. Suppose $v_1, v_2, \ldots, v_s\in N(v)\backslash N(u)$, $1\leq s\leq d(v)$. Let $G^*=G-\{vv_1, vv_2, \ldots, vv_s\}+\{uv_1, uv_2, \ldots, uv_s\}$. If $d(u)\geq d(v)$ and $u$ is not adjacent to $v$, then
\begin{align*}
\text{$M_1(G^*)>M_1(G)$\ \ \ \ and \ \ \ \ $M_2(G^*)>M_2(G).$}
\end{align*}
\end{lem}
Let $\mathcal{A}^k_n=\{G$: $G=(X, Y)\in \mathcal{V}^k_n$ and there exists a vertex cut set $S$ of order $k$ in $G$ such that $S\bigcap X\neq\emptyset$ and $S\bigcap Y\neq\emptyset$\}.
\begin{lem}\label{lem2.3}
If $G\in\mathcal{A}^k_n$, then $G$ cannot be the graph with the maximum $M_1(G)$ and $M_2(G)$ in $\mathcal{V}^k_n$.
\end{lem}
\begin{proof}
Choose $G=(X,Y)\in\mathcal{A}^k_n$ such that $M_1(G)$ and $M_2(G)$ are as large as possible. By the definition of $\mathcal{A}^k_n$, suppose that $S=\{v_1,v_2,\ldots,v_k\}$ is a vertex cut set of order $k$ in $G$ such that $S\bigcap X\neq\emptyset$ and $S\bigcap Y\neq\emptyset$. We write $S_X=S\bigcap X$, $S_Y=S\bigcap Y$, $S^c_X=X-S_X$ and $S^c_Y=Y-S_Y$.

First we show the following facts.
\begin{fact}
$G[S_X\bigcup S_Y]$ is a complete bipartite graph with partitions $(S_X,S_Y)$.
\end{fact}
{\bf Proof of Fact 1.}~
Suppose that, on the contrary, there exist $v_i\in S_X$ and $v_j\in S_Y$ such that $v_iv_j\notin E(G)$. By Lemma \ref{lem2.1}, we have $M_1(G+v_iv_j)>M_1(G)$ and $M_2(G+v_iv_j)>M_2(G)$, a contradiction, since $G+v_iv_j$ is also a bipartite graph with connectivity $k$.

This completes the proof of Fact 1.\qed

With the similar argument as in the proof of Fact 1, we can showing the following fact.
\begin{fact}
For every component $H$ of $G-S$, $G[V(H)\bigcup S]$ is a complete bipartite graph.
\end{fact}
\begin{fact}
$G-S$ consists of two complete bipartite graphs.
\end{fact}
{\bf Proof of Fact 3.}~
Suppose that, on the contrary, $G-S$ contains three connected components, say $H_1, H_2, H_3$. Now we show that $S^c_X\neq \emptyset$ and $S^c_Y\neq \emptyset$. Otherwise, if $S^c_X=\emptyset$, then $S_X=X$ is a vertex cut set of $G$ with order $|S_X|<k=\kappa(G)$, a contradiction. Similarly, we have $S^c_Y\neq \emptyset$. Without loss of generality, assume that $V(H_1)\bigcap X\neq\emptyset$ and $V(H_2)\bigcap Y\neq\emptyset$. Let $x\in V(H_1)\bigcap X$ and $y\in V(H_2)\bigcap Y$. For $H_3$, if there exists $z\in V(H_3)\bigcap X$, then we have $G+zy\in \mathcal{V}^k_n$. By Lemma \ref{lem2.1}, $M_1(G+zy)>M_1(G)$ and $M_2(G+zy)>M_2(G)$, a contradiction. Hence, there exists $z\in V(H_3)\bigcap Y$. It is easy to see that $G+zx\in \mathcal{V}^k_n$. By Lemma \ref{lem2.1}, $M_1(G+zx)>M_1(G)$ and $M_2(G+zx)>M_2(G)$, a contradiction. Hence, $G-S$ consists of two complete bipartite graphs.\qed

By Fact 3, $G-S$ consists of two complete bipartite graphs, say $H_1$ and $H_2$. Since $G$ is $k$-connected, we have $H_i\bigcap X\neq\emptyset$ and $H_i\bigcap Y\neq\emptyset$ for $i=1,2$. Let $x_1\in H_1\bigcap X$ and $x_2\in H_2\bigcap X$. By Claim 2, we have $N_G(x_1)=(V(H_1)\bigcap Y)\bigcup S_Y$ and $N_G(x_2)=(V(H_2)\bigcap Y)\bigcup S_Y$. If $d_G(x_1)\geq d_G(x_2)$, set $G^*=G-\{x_2y| y\in V(H_2)\bigcap Y\}+\{x_1y| y\in V(H_2)\bigcap Y\}$. By Lemma \ref{lem2.2}, we have $M_1(G^*)>M_1(G)$ and $M_2(G^*)>M_2(G)$. Clearly, $G^*$ is a bipartite graph with $d_{G^*}(x_2)= |S_Y|<k$. Note that $|Y|\geq k$, then we choose an arbitrary $(k-|S_Y|)$ subset of $V(H_1)\bigcap Y$, say $A$. Let
$$G'=G^*+\{x_2y: y\in A\}+\{uv: u\in X\backslash\{x_2\}, y\in Y\}.$$
It is routine to check that $G'\in \mathcal{V}^k_n\backslash \mathcal{A}^k_n$. By Lemma \ref{lem2.1}, we have $M_1(G')>M_1(G^*)>M_1(G)$ and $M_2(G')>M_2(G^*)>M_2(G)$. If $d_G(x_1)< d_G(x_2)$, then the discussion is similar, we omit its procedure here.

This completes the proof.
\end{proof}

\begin{lem}\label{lem2.4}
Choose $G=(X,Y)\in\mathcal{V}^k_n$ such that its $M_1(G)$, $M_2(G)$ are as large as possible. If $S$ is a vertex cut set of order $k$ in $G$ such that $S\subseteq X$ and $|X|>k$, then $G-S$ consists of two components, one of which is an isolated vertex.
\end{lem}
\begin{proof}
Firstly, we show that $G-S$ consists of two components. Suppose that $G-S$ contains three components, say $G_1, G_2, G_3$. Since $|X|>k$, then $G-S$ contains a component with at least two vertices. Without loss of generality, assume that $|V(G_1)|\geq2$. Hence $V(G_1)\bigcap X\neq\emptyset$ and $V(G_1)\bigcap Y\neq\emptyset$. Since $G$ is connected, then $V(G_2)\bigcap Y\neq\emptyset$. Let $x\in V(G_1)\bigcap X$ and $y\in V(G_2)\bigcap Y$. It is easy to see that $G+xy\in\mathcal{V}^k_n$. By Lemma \ref{lem2.1}, we have $M_1(G+xy)>M_1(G)$ and $M_2(G+xy)>M_2(G)$, a contradiction. Now we can assume that $G-S$ contains two components, say $G_1$ and $G_2$, such that $|V(G_1)|\geq |V(G_2)|$. If $G_2$ is not an isolated vertex, then we have $V(G_2)\bigcap X\neq\emptyset$. Let $y_1\in V(G_1)\bigcap Y$ and $y_2\in V(G_2)\bigcap Y$. If $d_G(y_2)\geq d_G(y_1)$, set $$\widetilde{G}=G-\Big\{y_1x: x\in V(G_1)\bigcap X\Big\}+\Big\{y_2x: x\in V(G_1)\bigcap X\Big\}+\Big\{uv: u\in X, y\in Y\backslash\{y_1\}\Big\}.$$
It is routine to check that $\widetilde{G}\in \mathcal{V}^k_n$. By Lemmas \ref{lem2.1} and \ref{lem2.2}, $M_1(\widetilde{G})>M_1(G)$ and $M_2(\widetilde{G})>M_2(G)$, a contradiction.

This complete the proof.
\end{proof}
\begin{Obs}\label{o}
Choose $G=(X,Y)\in \mathcal{V}^k_n$ such that its $M_1(G)$, $M_2(G)$ are as large as possible. Since $G$ is $k$-connected, then $|X|\geq k$ and $|Y|\geq k$. If $S$ be a vertex cut set of order $k$ in $G$, by Lemma \ref{lem2.3}, without loss of generality, we have $S\subseteq X$. If $|X|=k$, i.e., $S=X$, then $G\cong K_{k,n-k}$. If $|X|>k$, by Lemma \ref{lem2.4}, we have $G\cong O_k\vee_1(K_1\bigcup K_{n-r-1,r-k})$ for some $r>k$. In both cases, we have $G\cong O_k\vee_1(K_1\bigcup K_{n-r-1,r-k})$ for some $r\geq k$. We partition $V(G)$ into $\{v\}\bigcup C\bigcup A\bigcup B\bigcup \{a_{n-r-1}\}$, where $C=\{c_1, c_2, \ldots,c_k\}$, $A=\{a_1, a_2, \ldots,a_{n-r-2}\}$ and $B=\{b_1, b_2, \ldots,b_{r-k}\}$, see Fig. 1.
\end{Obs}
\begin{figure}[h!]
\begin{center}
\psfrag{a}{$v$}\psfrag{b}{$c_1$}\psfrag{c}{$c_2$}\psfrag{d}{$c_k$}
\psfrag{e}{$b_1$}\psfrag{f}{$b_2$}\psfrag{g}{$b_{r-k}$}
\psfrag{h}{$a_1$}\psfrag{i}{$a_2$}\psfrag{j}{$a_{n-r-1}$}
\includegraphics[width=60mm]{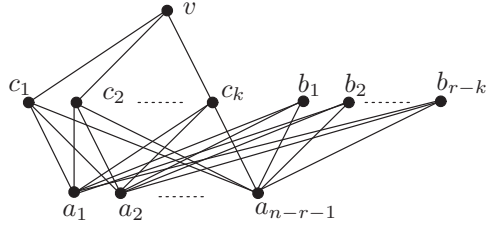} \\
  \caption{$G\cong O_k\vee_1(K_1\bigcup K_{n-r-1,r-k})$, $r\geq k$.} 
\end{center}
\end{figure}

\begin{lem}\label{2.5}
If $G\cong O_k\vee_1(K_1\bigcup K_{n-r-1,r-k})\in \mathcal{V}^k_n$ ($n\geq6$ and $k\leq\frac{n}{2}-1$) such that $r\leq \frac{n}{2}-2$ or $r>\frac{n}{2}$, then $G$ cannot be the graph with the maximum $M_1(G)$ and $M_2(G)$.
\end{lem}
\begin{proof}

We consider the following two cases.
\begin{Case}
$r\leq \frac{n}{2}-2$.
\end{Case}
In this case, let
$$G^*=G-\{a_{n-r-1}c_1, \ldots,a_{n-r-1}c_k, a_{n-r-1}b_1, \ldots, a_{n-r-1}b_{r-k}\}+\{a_{n-r-1}a_1, \ldots, a_{n-r-1}a_{n-r-2}\}.$$
It is rountine to check that $G^*\in \mathcal{V}^k_n$; see Fig. 2. And we have
\begin{figure}[h!]
\begin{center}
\psfrag{a}{$v$}\psfrag{b}{$c_1$}\psfrag{c}{$c_2$}\psfrag{d}{$c_k$}
\psfrag{e}{$b_1$}\psfrag{f}{$b_2$}\psfrag{g}{$b_{r-k}$}\psfrag{o}{$a_{n-r-1}$}
\psfrag{h}{$a_1$}\psfrag{i}{$a_2$}\psfrag{j}{$a_{n-r-2}$}
\psfrag{1}{$v$}\psfrag{2}{$c_1$}\psfrag{3}{$c_2$}\psfrag{4}{$c_k$}
\psfrag{5}{$b_1$}\psfrag{6}{$b_2$}\psfrag{7}{$b_{r-k}$}\psfrag{8}{$a_1$}
\psfrag{9}{$a_2$}\psfrag{k}{$a_k$}\psfrag{l}{$a_{n-r-1}$}
\psfrag{m}{$G^*$}\psfrag{n}{$G^{**}$}
\includegraphics[width=120mm]{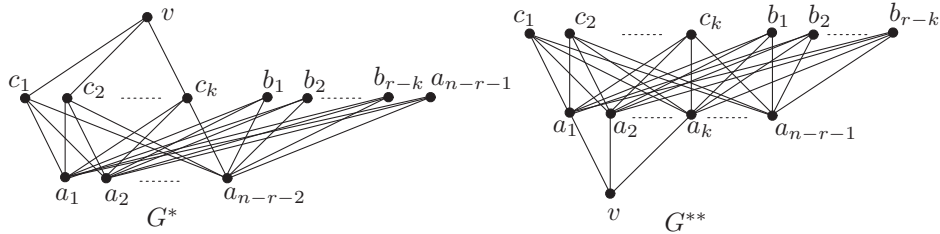} \\
  \caption{Graphs $G^*$ and $G^{**}$ used in Case 1 of Lemma 2.5.} 
\end{center}
\end{figure}
\begin{align*}
&M_1(G^*)-M_1(G)\\
=&\sum_{i=1}^{k}\Big[d^2_{G^*}(c_i)-d^2_{G}(c_i)\Big]+\sum_{i=1}^{r-k}\Big[d^2_{G^*}(b_i)-d^2_{G}(b_i)\Big]+\sum_{i=1}^{n-r-1}\Big[d^2_{G^*}(a_i)-d^2_{G}(a_i)\Big]\\
=&k[1-2(n-r)]+(r-k)[1-2(n-r-1)]+(n-r-2)(1+2r)+(n-r-2)^2-r^2\\
=&k[1-2(n-r)]+(r-k)[1-2(n-r)]+2(r-k)+(n-r-2)(1+2r)+(n-r-2)^2-r^2\\
=&r[1-2(n-r)]+2(r-k)+(n-r-2)(1+2r)+(n-r-2)^2-r^2\\
=&-3r+(n-r-2)+2(r-k)+(n-r-2)^2-r^2\\
=&[(n-r)-r-2]-2k+(n-r-2)^2-r^2\\
\geq&2-2k+(r+2)^2-r^2\\
=&-2k+4r+6\\
>&0,
\end{align*}
\begin{align*}
&M_2(G^*)-M_2(G)\\
=&\sum_{i=1}^{k}\Big[d_{G^*}(v)d_{G^*}(c_i)-d_{G}(v)d_{G}(c_i)\Big]+\sum_{i=1}^{k}\sum_{j=1}^{n-r-2}\Big[d_{G^*}(c_i)d_{G^*}(a_j)-d_{G}(c_i)d_{G}(a_j)\Big]\\
&+\sum_{i=1}^{r-k}\sum_{j=1}^{n-r-2}\Big[d_{G^*}(b_i)d_{G^*}(a_j)-d_{G}(b_i)d_{G}(a_j)\Big]+d_{G^*}(a_{n-r-1})\sum_{i=1}^{n-r-2}d_{G^*}(a_i)\\
&-d_{G}(a_{n-r-1})\Big[\sum_{i=1}^{k}d_{G}(c_i)+\sum_{i=1}^{r-k}d_G(b_i)\Big]\\
=&-k^2+(n-2r-1)(n-r-2)k+(n-2r-2)(r-k)(n-r-2)+(n-r-2)^2(r+1)\\
&-r^2(n-r)+r^2-rk\\
=&-k^2+(n-2r-1)(n-r-2)k+(n-2r-2)(r-k)(n-r-2)+\Big[(n-r)^2+4-4(n-r)\Big](r+1)\\
&-r^2(n-r)+r^2-rk\\
=&-k^2+(n-2r-1)(n-r-2)k+(n-2r-2)(r-k)(n-r-2)+\Big[(n-r)^2+4-4(n-r)\Big]\\
&+\Big[r(n-r)^2+4r-4(n-r)r\Big]-r^2(n-r)+r^2-rk\\
=&-k^2+(n-2r-1)(n-r-2)k+(n-2r-2)(r-k)(n-r-2)\\
&+\Big[(n-r)^2+4-4(n-r)+4r-4(n-r)r+r^2-rk\Big]+r(n-r)(n-r-r)\\
>&-k^2+\Big[(n-r)^2+4-4(n-r)+4r-4(n-r)r+r^2-rk\Big]+4r(n-r)\ \ \ \ \ \ \ \ \ \ \ \ \ \ \ \ \ \ \ \ \ \ \text{(by $n-2r\geq4$)}\\
=&-k^2+\Big[(n-r)^2+4-4(n-r)+4r+r^2-rk\Big]\\
=&-k^2+(n-r-2)^2+4r+r^2-rk\\
>&0.
\end{align*}

\begin{Case}
$r>\frac{n}{2}$.
\end{Case}
In this case, let
$$G^{**}=G-\{vc_1, vc_2, \ldots, vc_k\}+\{va_1, va_2, \ldots, va_k\}.$$
It is routine to check that $G^{**}\in \mathcal{V}^k_n$; see Fig. 2. And we have
\begin{align*}
&M_1(G^{**})-M_1(G)\\
=&\sum_{i=1}^{k}\Big[d^2_{G^{**}}(c_i)-d^2_{G}(c_i)\Big]+\sum_{i=1}^{r-k}\Big[d^2_{G^{**}}(b_i)-d^2_{G}(b_i)\Big]+\sum_{i=1}^{k}\Big[d^2_{G^{**}}(a_i)-d^2_{G}(a_i)\Big]+\sum_{i=k+1}^{n-r-1}\Big[d^2_{G^{**}}(a_i)-d^2_{G}(a_i)\Big]\\
=&k\Big[1-2(n-r)\Big]+k(2r+1)\\
=&k(2+4r-2n)\\
>&0,
\end{align*}
\begin{align*}
&M_2(G^{**})-M_2(G)\\
=&\sum_{i=1}^{k}\Big[d_{G^{**}}(v)d_{G^{**}}(a_i)-d_{G}(v)d_{G}(c_i)\Big]+\sum_{i=1}^{k}\sum_{j=1}^{k}\Big[d_{G^{**}}(c_i)d_{G^{**}}(a_j)-d_{G}(c_i)d_{G}(a_j)\Big]\\
&+\sum_{i=1}^{k}\sum_{j=k+1}^{n-r-1-k}\Big[d_{G^{**}}(c_i)d_{G^{**}}(a_j)-d_{G}(c_i)d_{G}(a_j)\Big]+\sum_{i=1}^{r-k}\sum_{j=1}^{k}\Big[d_{G^{**}}(b_i)d_{G^{**}}(a_j)-d_{G}(b_i)d_{G}(a_j)\Big]\\
=&k^2\Big[r+1-(n-r)\Big]+k^2\Big[(r+1)(n-r-1)-r(n-r)\Big]-rk(n-r-1-k)+(n-r-1)(r-k)k\\
=&k^2\Big[r+1-(n-r)\Big]+k^2\Big[(n-r)-r-1\Big]-rk(n-r-1)+rk^2+rk(n-r)-rk-(n-r)k^2+k^2\\
=&rk^2-(n-r)k^2+k^2\\
=&(2r-n+1)k^2\\
>&0.
\end{align*}

This completes the proof.
\end{proof}

\begin{thm}\label{2.6}
Let $G$ be the graph in $\mathcal{V}^k_n$ ($n\geq6$) with the maximum $M_1(G)$ (resp. $M_2(G)$). Then we have the following results
\begin{wst}
\item[{\rm (i)}]$G\cong O_k\vee_1(K_1\bigcup K_{\frac{n-1}{2},\frac{n-1-2k}{2}})$ if $n$ is odd,
\item[{\rm (ii)}]$G\cong K_{\frac{n}{2}, \frac{n}{2}}$ if $n$ is even and $k=\frac{n}{2}$; $G\cong O_k\vee_1(K_1\bigcup K_{\frac{n}{2},\frac{n-2-2k}{2}})$ if $n$ is even and $k\leq\frac{n}{2}-1$.
\end{wst}
\end{thm}
\begin{proof}
Since $G$ is bipartite, by the definition of $k$-connected, we have $k\leq\frac{n}{2}$. If $k=\frac{n-1}{2}$ (resp. $k=\frac{n}{2}$), it is easy to see that $\mathcal{V}^k_n=\{K_{\frac{n-1}{2}, \frac{n+1}{2}}\}$ (resp. $\mathcal{V}^k_n=\{K_{\frac{n}{2}, \frac{n}{2}}\}$) and the result is clear. In what follows, assume that $k\leq\frac{n}{2}-1$. By Observation \ref{o} and Lemma \ref{2.5}, we have $G\in \Big\{O_k\vee_1(K_1\bigcup K_{n-r-1,r-k}): \frac{n-3}{2}\leq r\leq\frac{n}{2}\Big\}\subseteq \mathcal{V}^k_n$. We consider the following two cases.
\setcounter{Case}{0}
\begin{Case}
$n$ is odd.
\end{Case}
In this case, we have
$$
\Big\{O_k\vee_1(K_1\bigcup K_{n-r-1,r-k})| \frac{n-3}{2}\leq r\leq\frac{n}{2}\Big\}=\Big\{O_k\vee_1(K_1\bigcup K_{\frac{n+1}{2},\frac{n-3-2k}{2}}), O_k\vee_1(K_1\bigcup K_{\frac{n-1}{2},\frac{n-1-2k}{2}})\Big\}.
$$
Let $G\cong O_k\vee_1(K_1\bigcup K_{\frac{n-1}{2},\frac{n-1-2k}{2}})$ and $G'\cong O_k\vee_1(K_1\bigcup K_{\frac{n+1}{2},\frac{n-3-2k}{2}})$, see Fig. 3. Then we have
\begin{figure}[h!]
\begin{center}
\psfrag{a}{$v$}\psfrag{b}{$c_1$}\psfrag{c}{$c_2$}\psfrag{d}{$c_k$}
\psfrag{e}{$b_1$}\psfrag{f}{$b_2$}\psfrag{g}{$b_{\frac{n-1-2k}{2}}$}
\psfrag{h}{$a_1$}\psfrag{i}{$a_2$}\psfrag{j}{$a_{\frac{n-1}{2}}$}
\psfrag{1}{$v$}\psfrag{2}{$c_1$}\psfrag{3}{$c_2$}\psfrag{4}{$c_k$}
\psfrag{5}{$b_1$}\psfrag{6}{$b_2$}\psfrag{7}{$b_{\frac{n-3-2k}{2}}$}\psfrag{8}{$a_1$}
\psfrag{9}{$a_2$}\psfrag{k}{$a_{\frac{n-1}{2}}$}\psfrag{l}{$b_{\frac{n-1-2k}{2}}$}
\psfrag{m}{$G$}\psfrag{n}{$G'$}
\includegraphics[width=120mm]{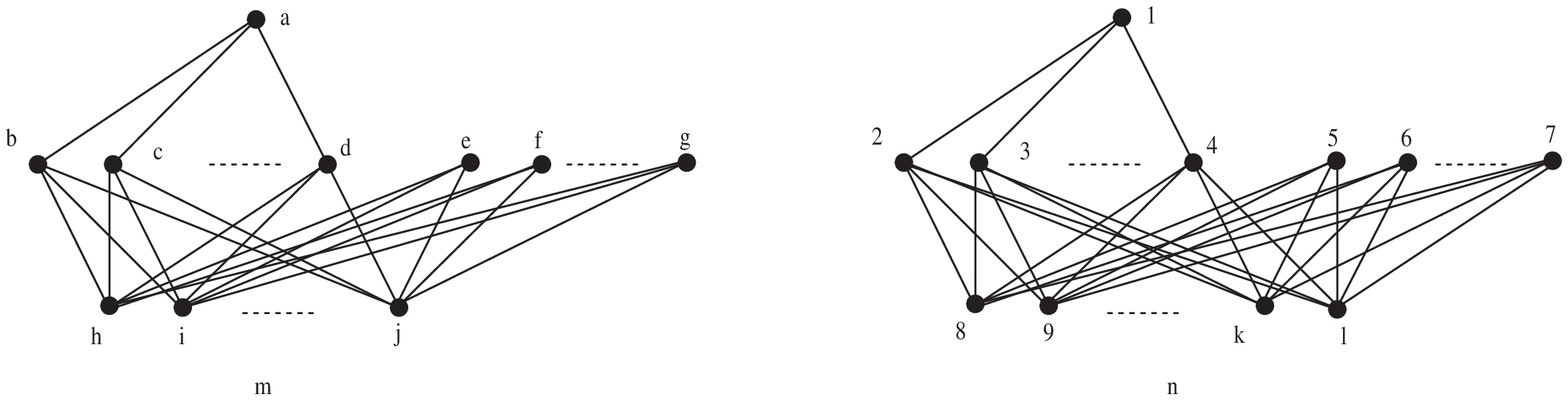} \\
  \caption{Graphs $G$ and $G'$ used in Case 1 of Theorem 2.6.} 
\end{center}
\end{figure}

\begin{align*}
&M_1(G)-M_1(G')\\
=&\sum_{i=1}^{k}\Big[d^2_{G}(c_i)-d^2_{G'}(c_i)\Big]+\sum_{i=1}^{\frac{n-1-2k}{2}}\Big[d^2_{G}(b_i)-d^2_{G'}(b_i)\Big]+\sum_{i=1}^{\frac{n-1}{2}}\Big[d^2_{G}(a_i)-d^2_{G'}(a_i)\Big]\\
=&k\Big[(\frac{n+1}{2})^2-(\frac{n+3}{2})^2\Big]+\frac{n-3-2k}{2}\Big[(\frac{n-1}{2})^2-(\frac{n+1}{2})^2\Big]+\frac{n-1}{2}\Big[(\frac{n-1}{2})^2-(\frac{n-3}{2})^2\Big]\\
&+(\frac{n-1}{2})^2-(\frac{n-3}{2})^2\\
=&\frac{k}{4}(-4n-8)-\frac{2n^2-6n-4nk}{4}+\frac{n-1}{4}(2n-4)+\frac{n^2+1-2n}{4}-\frac{n^2+9-6n}{4}\\
=&-2k+n-1\\
>&(-2)\times\frac{n-1}{2}+n-1\ \ \ \ \ \ \ \ \ \ \ \ \ \ \ \ \ \ \ \ \ \ \text{(by $k\leq\frac{n}{2}-1$)}\\
=&0,
\end{align*}
\begin{align*}
&M_2(G)-M_2(G')\\
=&\sum_{i=1}^{k}\Big[d_{G}(v)d_{G}(c_i)-d_{G'}(v)d_{G'}(c_i)\Big]+\Big[\sum_{i=1}^{k}\sum_{j=1}^{\frac{n-1}{2}}d_{G}(c_i)d_{G}(a_j)-\sum_{i=1}^{k}\sum_{j=1}^{\frac{n-1}{2}}d_{G'}(c_i)d_{G'}(a_j)\\
&-d_{G'}(b_{\frac{n-1-2k}{2}})\sum^{k}_{i=1}d_{G'}(c_i)\Big]+\Big[\sum_{i=1}^{\frac{n-1-2k}{2}}\sum_{j=1}^{\frac{n-1}{2}}d_{G}(b_i)d_{G}(a_j)-\sum_{i=1}^{\frac{n-3-2k}{2}}\sum_{j=1}^{\frac{n-1}{2}}d_{G'}(b_i)d_{G'}(a_j)\Big]\\
=&-k^2+\Big[\frac{n+1}{2}\cdot\frac{n-1}{2}\cdot k\cdot\frac{n-1}{2}-\frac{n+3}{2}\cdot\frac{n-3}{2}\cdot k\cdot\frac{n+1}{2}\Big]+\Big[\frac{n-1}{2}\cdot\frac{n-1}{2}\cdot\frac{n-1-2k}{2}\cdot\frac{n-1}{2}\\
&-\frac{n+1}{2}\cdot\frac{n-3}{2}\cdot\frac{n-3-2k}{2}\cdot\frac{n+1}{2}\Big]\\
=&-k^2+\frac{k(n+1)}{8}\cdot(10-2n)+\frac{1}{8}\cdot(-8n+4n^2-4-8nk+2kn^2-2k)\\
=&\frac{1}{2}\Big[-2n+n^2-1+2k-2k^2\Big]\\
\geq&\frac{1}{2}\Big[-2n+n^2-1+2k-2(\frac{n}{2}-1)^2\Big]\\
=&\frac{1}{2}(\frac{n^2}{2}+2k-3)\\
>&0.
\end{align*}
Thus, we have (i) is true.

\begin{Case}
$n$ is even.
\end{Case}
In this case, we have
$$
\Big\{O_k\vee_1(K_1\bigcup K_{n-r-1,r-k})| \frac{n-3}{2}\leq r\leq\frac{n}{2}\Big\}=\Big\{O_k\vee_1(K_1\bigcup K_{\frac{n-2}{2},\frac{n-2k}{2}}), O_k\vee_1(K_1\bigcup K_{\frac{n}{2},\frac{n-2-2k}{2}})\Big\}.
$$
Let $G\cong O_k\vee_1(K_1\bigcup K_{\frac{n-2}{2},\frac{n-2k}{2}})$ and $G'\cong O_k\vee_1(K_1\bigcup K_{\frac{n}{2},\frac{n-2-2k}{2}})$, see Fig. 4. Then we have
\begin{figure}[h!]
\begin{center}
\psfrag{a}{$v$}\psfrag{b}{$c_1$}\psfrag{c}{$c_2$}\psfrag{d}{$c_k$}
\psfrag{e}{$b_1$}\psfrag{f}{$b_2$}\psfrag{g}{$b_{\frac{n-2-2k}{2}}$}
\psfrag{h}{$a_1$}\psfrag{i}{$a_2$}\psfrag{j}{$a_{\frac{n-2}{2}}$}
\psfrag{1}{$v$}\psfrag{2}{$c_1$}\psfrag{3}{$c_2$}\psfrag{4}{$c_k$}
\psfrag{5}{$b_1$}\psfrag{6}{$b_2$}\psfrag{7}{$b_{\frac{n-2-2k}{2}}$}\psfrag{8}{$a_1$}
\psfrag{9}{$a_2$}\psfrag{k}{$a_{\frac{n-2}{2}}$}\psfrag{l}{$a_{\frac{n}{2}}$}\psfrag{o}{$a_{\frac{n}{2}}$}
\psfrag{n}{$G$}\psfrag{m}{$G'$}
\includegraphics[width=120mm]{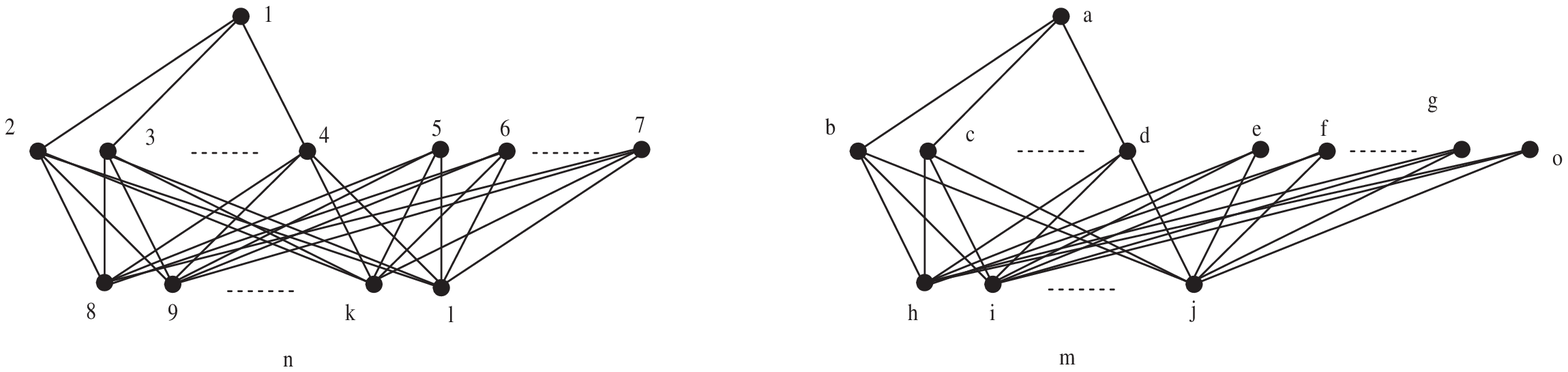} \\
  \caption{Graphs $G$ and $G'$ used in Case 2 of Theorem 2.6.} 
\end{center}
\end{figure}
\begin{align*}
&M_1(G)-M_1(G')\\
=&\sum_{i=1}^{k}\Big[d^2_{G}(c_i)-d^2_{G'}(c_i)\Big]+\sum_{i=1}^{\frac{n-2-2k}{2}}\Big[d^2_{G}(b_i)-d^2_{G'}(b_i)\Big]+\sum_{i=1}^{\frac{n}{2}}\Big[d^2_{G}(a_i)-d^2_{G'}(a_i)\Big]\\
=&k\Big[(\frac{n+2}{2})^2-(\frac{n}{2})^2\Big]+\frac{n-2-2k}{2}\Big[(\frac{n}{2})^2-(\frac{n-2}{2})^2\Big]+\frac{n-2}{2}\Big[(\frac{n-2}{2})^2-(\frac{n}{2})^2\Big]\\
&++(\frac{n-2}{2})^2-(\frac{n-2}{2})^2\\
=&2k\\
>&0,
\end{align*}
\begin{align*}
&M_2(G)-M_2(G')\\
=&\sum_{i=1}^{k}\Big[d_{G}(v)d_{G}(c_i)-d_{G'}(v)d_{G'}(c_i)\Big]+\Big[\sum_{i=1}^{k}\sum_{j=1}^{\frac{n}{2}}d_{G}(c_i)d_{G}(a_j)-\sum_{i=1}^{k}\sum_{j=1}^{\frac{n}{2}}d_{G'}(c_i)d_{G'}(a_j)\Big]\\
&+\Big[\sum_{i=1}^{\frac{n-2-2k}{2}}\sum_{j=1}^{\frac{n}{2}}d_{G}(b_i)d_{G}(a_j)-\sum_{i=1}^{\frac{n-2-2k}{2}}\sum_{j=1}^{\frac{n-2}{2}}d_{G'}(b_i)d_{G'}(a_j)-d_{G'}(a_{\frac{n}{2}})\sum^{\frac{n-2}{2}}_{i=1}d_{G'}(a_i)\Big]\\
=&k^2+\Big[\frac{n+2}{2}\cdot\frac{n-2}{2}\cdot k\cdot\frac{n}{2}-\frac{n}{2}\cdot\frac{n}{2}\cdot k\cdot\frac{n-2}{2}\Big]+\Big[\frac{n-2-2k}{2}\cdot\frac{n}{2}\cdot\frac{n-2}{2}\cdot\frac{n}{2}\\
&-\frac{n-2k}{2}\cdot\frac{n-2}{2}\cdot\frac{n-2}{2}\cdot\frac{n}{2}\Big]\\
=&k^2+\frac{n-2}{2}\cdot\frac{n}{2}\cdot k-k\cdot\frac{n-2}{2}\cdot\frac{n}{2}\\
=&k^2\\
>&0.
\end{align*}
Thus, we have (ii) is true.

This completes the proof.
\end{proof}
\begin{thm}
Let $G$ be the graph in $\mathcal{E}^s_n$ ($n\geq6$) with the maximum $M_1(G)$ (resp. $M_2(G)$). Then we have the following results
\begin{wst}
\item[{\rm (i)}]$G\cong O_s\vee_1(K_1\bigcup K_{\frac{n-1}{2},\frac{n-1-2s}{2}})$ if $n$ is odd,
\item[{\rm (ii)}]$G\cong K_{\frac{n}{2}, \frac{n}{2}}$ if $n$ is even and $s=\frac{n}{2}$; $G\cong O_s\vee_1(K_1\bigcup K_{\frac{n}{2},\frac{n-2-2s}{2}})$ if $n$ is even and $s\leq\frac{n}{2}-1$.
\end{wst}
\end{thm}
\begin{proof}
By the inequality $\kappa(G)\leq\kappa'(G)\leq \delta(G)$, we have $G\in\bigcup_{k=1}^s\mathcal{V}^k_n$. Since $G$ is a bipartite graph, we have $\delta(G)\leq\frac{n}{2}$. Combining $\kappa(G)\leq\kappa'(G)\leq \delta(G)$, Lemma \ref{lem2.1} and Theorem \ref{2.6}, we have
\begin{wst}
\item[{\rm (i)}]$G\cong O_s\vee_1(K_1\bigcup K_{\frac{n-1}{2},\frac{n-1-2s}{2}})$ if $n$ is odd,
\item[{\rm (ii)}]$G\cong K_{\frac{n}{2}, \frac{n}{2}}$ if $n$ is even and $s=\frac{n}{2}$; $G\cong O_s\vee_1(K_1\bigcup K_{\frac{n}{2},\frac{n-2-2s}{2}})$ if $n$ is even and $s\leq\frac{n}{2}-1$.
\end{wst}
\end{proof}

This completes the proof.

\end{document}